\documentclass[12pt]{amsart}

\usepackage[inner=3cm,outer=2.5cm,bottom=3cm,top=3cm]{geometry}

\usepackage{setspace}
\usepackage[latin1]{inputenc}
\usepackage[T1]{fontenc}
\usepackage{amsfonts}
\usepackage{amsmath}
\usepackage{amssymb}
\usepackage{amsthm}
\usepackage[english]{babel}
\usepackage{hyperref, url}
\usepackage{enumerate}
\usepackage{color}
\usepackage{tikz}
\usepackage{lmodern}
\usepackage{epstopdf}

\newtheorem{thm}{Theorem}[section]

\theoremstyle{remark}

\newtheorem{rem}[thm]{Remark}

\numberwithin{equation}{section}

\newcommand{\C}{\mathbb{C}}
\newcommand{\CP}{\mathbb{CP}}

\renewcommand{\epsilon}{\varepsilon}

\title{Construction of bicuspidal rational complex projective plane curves}

\author{J\'{o}zsef Bodn\'{a}r}
\address{A. R\'enyi Institute of Mathematics, 1053 Budapest,
Re\'altanoda u. 13-15,  Hungary.}
\email{bodnar.jozef@renyi.mta.hu}
\thanks{The author is supported by the young researcher program and
 the ERC program `LTDBud' at MTA Alfr\'ed R\'enyi Institute of Mathematics.}

\date{}

\begin{document}

\begin{abstract}
We give several constructions of bicuspidal rational complex projective plane curves, and list the Newton pairs and the multiplicity sequences of the singularities on the resulting curves.
\end{abstract}

\maketitle


\pagestyle{myheadings} \markboth{{\normalsize
J. Bodn\'ar}}{ {\normalsize Construction of bicuspidal rational complex projective plane curves}}

\section{Introduction}

\subsection{Introduction} The study of possible singularity types on complex projective plane curves has a long history.
Just to mention a few results, Namba in~\cite{NamGeo} classified all projective plane curves up to degree $5$.
Flenner and Zaidenberg listed the possible local cusp types (types of locally irreducible singularities) of tricuspidal curves with $d-m = 2$ in~\cite{FleZaiOna} and with $d-m = 3$ in~\cite{FleZaiRat}, where $d$ is the degree of the curve and $m$ is the maximal multiplicity among the multiplicities of the cusps on the curve.
The work in this direction was continued by Fenske, who classified unicuspidal and bicuspidal curves with $d-m = 2$ and $3$ in~\cite{FenRat1}.
In~\cite{FenRat}, he discussed the case of $d-m = 4$ under some additional constraints.

In these works, constructing rational cuspidal curves by giving explicit equations and birational transformations of the complex projective plane plays a crucial role.
One method is to start with a conic in $\mathbb{CP}^2$ and construct a series of birational transformations leading to a series of cuspidal curves.
The birational transformation is usually given by a configuration of divisors (with given self-intersection numbers) on a certain complex surface, such that these divisors  can be blown down in two different ways to get $\mathbb{CP}^2$ (see for example the divisor configurations on Figures~\ref{fig:cremona0}, \ref{fig:cremona1} and \ref{fig:cremona2} in the proof of Theorem~\ref{thm:main}).

In this note, we would like to give some examples of curve constructions leading to bicuspidal rational complex projective plane curves.
To the best of the knowledge of the author, curves under \textbf{(b)} and \textbf{(c)} in Theorem~\ref{thm:main} are new, others are (implicitly) present in \cite{BorZolCom} (see Remark~\ref{rem:notice} and the proof of Theorem~\ref{thm:main}).

In~\cite{BorZolCom}, certain embeddings of $\mathbb{C}^{\ast}$ into $\mathbb{C}^2$ are listed.
Some of them lead to cuspidal complex projective plane curves having exactly two intersection points with the line at infinity in $\mathbb{CP}^2$.
We will be interested only in bicuspidal curves in this note.

Although the existence of cusp configurations \textbf{(a1)--(a4), (d1), (d2), (e), (f)} listed in Theorem~\ref{thm:main} is already proved in \cite{BorZolCom}, the advantage of our approach is that on one hand we explicitly state the Newton pairs and the multiplicity sequences of the singularities on the curves, on the other hand, we also give birational transformations transforming a conic or a bicuspidal curve with simple equation to these rational bicuspidal curves with given cusp configurations.
This explicit construction is left out from \cite{BorZolCom}, although the authors checked that all the given curves are rectifiable, i.e. there is a birational transformation of $\mathbb{CP}^2$ transforming the curve into a line (see \cite[Commentary, p. 306]{BorZolCom}).

Notice that from our construction, it is clear that curves \textbf{(a1)--(a4)} are common generalizations of some curves constructed by Fenske in \cite{FenRat1} and Tono in \cite{TonOna} (see Remark~\ref{rem:notice}).

\subsection{Organization of the paper} After an overview of some relevant invariants of embedded local topological types of complex plane curve singularities and basic facts about blow-ups and birational transformations in Section~\ref{sec:notation}, we give the main theorem -- a list of some existing cusp configurations of bicuspidal curves -- in Section~\ref{sec:construction}, together with the proof by explicit construction of these curves with the given cusp configurations.

\subsection{Acknowledgement} The author would like to thank Andr\'as N\'emethi for useful discussions, the construction of series \textbf{(b)} and the equations of the initial curves in the construction of series \textbf{(a1)--(a4)} in the proof of Theorem~\ref{thm:main}.
The author is grateful to Maciej Borodzik for several discussions and advice.

\section{Notation and definitions}\label{sec:notation}

\subsection{Local plane curve singularities}

By a \emph{local plane curve singularity} we mean a holomorphic function germ $f \in \mathbb{C}\{x,y\}$ such that $f(\mathbf{0}) = 0$, where $\mathbf{0} = (0,0) \in \C^2$ is the origin.
Strictly speaking, we say it is singular only if $\frac{\partial f}{\partial x}(\mathbf{0}) = \frac{\partial f}{\partial y}(\mathbf{0}) = 0$.
We say it is \emph{locally irreducible} if $f$ is irreducible in $\mathbb{C}\{x,y\}$.
We will deal only with locally irreducible singularities in this work, and will often use the word \emph{cusp} for them.
For a sufficiently small neighborhood $B$ of the origin in $\C^2$ (see \cite[Lemma 5.2.1]{WalSin}, in particular, the derivative of $f$ does not vanish in $B \setminus \mathbf{0}$) we will call the zero set defined by $\{(x,y) \in B : f(x,y) = 0\}$ a singular set, a singular \emph{curve branch} or \emph{singularity}, with a slight abuse of terminology.

Two singularities defined by functions $f_1$ and $f_2$ are said to be \emph{locally topologically equivalent} if there exist neighborhoods $B_1, B_2 \ni \mathbf{0}$ of the origin of $\mathbb{C}^2$ and a homeomorphism $\varphi: B_1 \rightarrow B_2$ between them such that its restriction induces a homeomorphism between the zero sets $\{f_1 = 0\} \cap B_1$ and $\{f_2 = 0\} \cap B_2$ inside neighborhoods as well.
We say that two plane curve singularities have the same \emph{local embedded topological type} (or \emph{same topological type} for short) if they are locally topologically equivalent.

Although we allow $f$ to be a power series, in fact, one can assume $f$ is a polynomial: every local plane curve singularity is equivalent topologically (and even analytically) to a singularity given by a polynomial $f \in \mathbb{C}[x,y]$ (see e.g.~\cite[\S 8.3, Theorem 15]{BriKnoPla},~\cite[\S 2.2]{GLSInt}).

Every irreducible local singularity defined by a function $f \in \C\{x,y\}$ can be parametrized by some power series $x(t), y(t) \in \C\{t\}$. 
The \emph{local parametrization} means that $f(x(t),y(t)) \equiv 0$ and $t \mapsto (x(t), y(t))$ is a bijection between some neighborhood of the origin in $\C$ and a neighborhood of the origin in the zero set of $f$.
Up to local topological equivalence, we can assume that the first power series is of form $x(t) = t^a$ for some integer $a > 1$ and the other power series is in fact a polynomial (see e.g.~\cite[Chapter I, Section 3, Theorem 3.3]{GLSInt},~\cite[\S 8.3]{BriKnoPla}): every local plane curve singularity is topologically equivalent to a singularity having local parametrization given by
\begin{equation}\label{eq:localparam}
 x(t) = t^a, \quad y(t) = t^{b_1} + t^{b_2} + \dots + t^{b_r},
\end{equation}
where $a < b_1 < \dots < b_r$, $a > \textrm{gcd}(a,b_1) > \textrm{gcd}(a,b_1,b_2) > \dots > \textrm{gcd}(a,b_1,b_2,\dots, b_r) = 1$.

In the above case we say that the \emph{multiplicity} of the singularity is $a$ and the line $\{y=0\}$ is the \emph{tangent line} of the singularity.
The multiplicity is an invariant of the local topological type.
The numbers $\{a; b_1, \dots, b_r\}$ are called the \emph{characteristic exponents} of the singularity.

We can write formally $y$ as a polynomial of $t = x^{1/a}$, leading to fractional powers (not necessarily in reduced form):
\[ y = x^{\frac{b_1}{a}} +  x^{\frac{b_2}{a}} + \dots +  x^{\frac{b_r}{a}}. \]

One can rewrite the above fractional power sum as
\begin{equation}\label{eq:newtonparam}
y = x^{\frac{q_1}{p_1}}\left( 1 + x^{\frac{q_2}{p_1 p_2}}\left( \dots  \left( 1 +  x^{\frac{q_r}{p_1 p_2 \dots p_r}} \right) \dots \right) \right),
\end{equation}
where  $p_1 < q_1$ and $p_i, q_i$ are coprime for all $i = 1, \dots, r$.
The pairs of numbers 
\[(p_1, q_1), (p_2, q_2), \dots, (p_r, q_r)\]
are called the \emph{Newton pairs} of the singularity.

The \emph{multiplicity sequence} of the singularity (see~\cite[\S 1]{FenRat},~\cite[\S 3.5]{WalSin}) can be defined inductively as follows.
Assume that $f \in \mathbb{C}[x,y]$ is a representative of the given topological type which admits a parametrization given by~\eqref{eq:localparam}. 
Then we can write
\[ f(x_1, x_1y_1) = x_1^a f_1(x_1,y_1) \]
for some $f_1 \in \mathbb{C}[x,y]$ such that $x_1$ does not divide $f_1$. 
Then the multiplicity sequence of $f$ is $[m_1, m_2, \dots, m_s]$, where $m_1 = a$ and $[m_2, \dots, m_s]$ is the multiplicity sequence of the topological type of the singularity given by equation $f_1(x, y) - f_1(0,0) = 0$. 
The multiplicity sequence of a smooth local branch is by definition the empty sequence $[\ ]$. 
It can be proved that the multiplicity sequence is well-defined and $m_1 \geq m_2 \geq \dots \geq m_s > 1$. 
For brevity, we will often write $n_k$ for $k$ consecutive copies of the number $n$ in a multiplicity sequence.
For example, we can write $[4_2, 2_3]$ instead of $[4,4,2,2,2]$.

\subsection{Rational cuspidal curves}

A complex projective plane curve is given as a zero set in the complex projective plane of an irreducible homogeneous polynomial in three variables with complex coefficients:
\[ C = \{[x:y:z] \ : \ h(x,y,z) = 0\} \subset \mathbb{CP}^2 \]
for some irreducible homogeneous polynomial $h \in \mathbb{C}[x,y,z]$.
The degree $d$ of the polynomial is called the \emph{degree of the curve}.

A point $P \in C$ is called \emph{singular} if the gradient vector of the defining equation vanishes at that point, i.e.
\[ \frac{\partial h}{\partial x}(P) = \frac{\partial h}{\partial y}(P) = \frac{\partial h}{\partial z}(P) = 0. \]
Denote by $\textrm{Sing\ } C = \{P_1, P_2, \dots, P_{\nu}\}$ the (finite) set of singular points. 
At each singular point $P \in \textrm{Sing\ } C$ the defining equation $h$ is a representative of a singular function germ.
Therefore, it determines a local plane curve singularity.
The tangent line of a curve $C$ at the point $p$ will be denoted by $T_pC$.

The curve $C$ is called \emph{cuspidal} if all of its singularities are locally irreducible plane curve singularities.
A curve is called \emph{bicuspidal} if $\nu = 2$, that is, it has two singularities only and those are locally irreducible.
The cuspidal curve $C$ is called \emph{rational} if it is homeomorphic to the two-dimensional real sphere $S^2$ (equivalently, if it can be parametrized by the complex projective line $\mathbb{CP}^1$).
By the \emph{degree-genus formula} (see e.g.~\cite[Section II.11]{BPVCom}) one obtains that $C$ is rational if and only if
\begin{equation}\label{eq:degreegenusrational}
\frac{(d-1)(d-2)}{2} = \sum_{j=1}^{\nu} \delta_j,
\end{equation}
where $\delta_j$ is the delta invariant of the singularity at $P_j$ (see e.g. \cite[\S 4.3]{WalSin}).

We will be interested in the question that what are the possible embedded local topological types of singularities on a bicuspidal rational complex projective plane curve.
We do not give a complete classification of possible cusp types, just give a list of existing configurations, providing some explicit constructions.

\subsection{Birational transformations}

The blow-up and blow-down process will be used in the proof of Theorem~\ref{thm:main} to construct cuspidal curves by Cremona transformations.
For further details on the blow-up process and Cremona transformations we refer to~\cite[\S 4, \S 5]{WalSin} and~\cite[\S 5]{MoeRat}.
We briefly recall here the most important facts.
The topological type of the singularity of the strict transform after a blow-up depends only on the topological type of the original singularity.
More concretely, blowing up a singular point with multiplicity sequence $[m_1, m_2, \dots, m_s]$ we obtain a singularity with multiplicity sequence $[m_2, \dots, m_s]$ (in particular, we get a smooth curve if the latter sequence is empty).

The \emph{local intersection multiplicity} (see \cite[\S 4]{WalSin}) of two local curve branches $C_1$ and $C_2$ at a point $p$ will be denoted by $(C_1 \cdot C_2)_p$.

If a smooth local curve branch had a local intersection multiplicity $m$ with the singular curve having multiplicity sequence $[m_1, m_2, \dots, m_s]$, after the blow-up, the strict transforms of the two curve branches will have local intersection multiplicity $m-m_1$ (in particular, they will be disjoint if $m=m_1$).
The new exceptional divisor will have intersection multiplicity $m_1$ with the strict transform of the curve, and self-intersection $-1$.
Blowing up a smooth point of any divisor decreases its self-intersection by $1$.
More generally, if a curve with self-intersection $e$ on a complex surface has a singular point with multiplicity $m$, then blowing up that point leads to a strict transform with self-intersection $e-m^2$.

For  simplicity, when describing a series  of blow-ups and blow-downs, we will use the same notation for a curve/branch and for its strict transform in the other surface.

There are several key invariants appearing in the results on projective plane curves. 
One can take the strict transform $\overline{C}$ under the \emph{local embedded minimal good resolution} $X \rightarrow \mathbb{CP}^2$ of the singularities of the cuspidal curve $C$. 
(That is, we blow up $\mathbb{CP}^2$ several times until we resolve the singularities and obtain a normal crossing configuration: the exceptional divisors and the strict transform of the curve, which is smooth, intersect each other transversely and no three of them goes through the same point.) 
We denote by $\overline{C}^2$ the self-intersection of this strict transform $\overline{C}$ in $X$.

If the cuspidal curve $C$ is of degree $d$ and has $\nu$ singularities with multiplicity sequences $[m_1^{(j)}, \dots, m_{r_j}^{(j)}]$, $j = 1, \dots, \nu$, then
\[ \overline{C}^2 = d^2 - \sum_{j=1}^{\nu} \sum_{i=1}^{r_j} (m_{i}^{(j)})^2 - \sum_{j=1}^{\nu} m_{r_j}^{(j)}. \]


\section{Construction of bicuspidal curves}\label{sec:construction}

\begin{thm}[{\emph{Partially based on Borodzik--\.{Z}o\l\k{a}dek,~\cite{BorZolCom}}}]\label{thm:main}
The following rational bicuspidal curves exist (the singularity types are given by multiplicity sequences and Newton pairs):
\begin{enumerate}

	
	\item[(a1)] A curve of degree $d = u(l-1)m + m - u + 1$ ($u \geq 2, l \geq 2, m \geq 2$) with the following cusp types:
	  \subitem  $[(u-1)(l-1)m + m - u + 1, (m(l-1)-1)_{u-1}, m_{l-2}, m-1]$,
	  \subitem  $[((l-1)m)_u, m_{l-1}]$; alternatively,
	  \subitem  $\left( (u-1)(l-1)m + m - u + 1, u(l-1)m + m - u\right)$,
	  \subitem  $(l-1, u(l-1) + 1)(m, 1)$.
	  
	\item[(a2)] A curve of degree $d = ulm + m + 1$ ($u \geq 2, l \geq 1, m \geq 2$) with the following cusp types:
	  \subitem  $[(lm+1)_u, m_l]$,
	  \subitem  $[(u-1)ml + m, (lm)_{u-1}, m_l]$; alternatively,
	  \subitem  $\left( lm+1, m(ul+1) + u\right)$,
	  \subitem  $((u-1)l+1, ul+1)(m, 1)$.
	  
	\item[(a3)]  A curve of degree $d = ulm + um + 1$ ($u \geq 2, l \geq 1, m \geq 2$) with the following cusp types:
	  \subitem  $[((l+1)m+1)_{u-1}, lm+1, m_l]$,
	  \subitem  $[(u-1)(l+1)m, ((l+1)m)_{u-1}, m_{l+1}]$; alterntively,
	  \subitem  $(m(l+1)+1, m(u(l+1)-1)+u)$,
	  \subitem  $(u-1,u)(l+1,1)(m,1)$.
	  
	\item[(a4)] A curve of degree $d = ulm - u + 1$ ($u \geq 2, l \geq 2, m \geq 2$) with the following cusp types:
	  \subitem  $[(u-1)(lm-1), (lm-1)_{u-1}, m_{l-1}, m-1]$,
	  \subitem  $[(lm)_{u-1}, (l-1)m, m_{l-1}]$; alternatively,
	  \subitem  $(u-1, u)(lm-1, m)$,
	  \subitem  $(l, ul-1)(m, 1)$.
	  
	\item[(b)] A curve of degree $d = 2k+1$, $k \geq 2$ with the following cusp types:
	  \subitem $[k_4]$ and $[2_k]$; alternatively,
	  \subitem $(k,4k+1)$ and $(2,2k+1)$.

	\item[(c)] A curve of degree $d = 4k+1$, $k \geq 1$ with the following cusp types:
	  \subitem $[(2k)_3, 2_k]$ and $[(2k), 2_k]$; alternatively,
	  \subitem $(k, 3k+1)(2,1)$ and $(k,k+1)(2,1)$.
	
    	\item[(d1)] A curve $C_k$ of degree $d = 8k+2$, $k \geq 1$ with the following cusp types:
	  \subitem $[4k+2, 4k-2, 4_{k-1}, 2_2]$ and $[(4k)_2, 4_k]$; alternatively,
	  \subitem $(2k+1, 4k)(2,1)$ and $(k, 2k+1)(4,1)$.

	\item[(d2)] A curve $D_k$ of degree $d = 8k+6$, $k \geq 1$ with the following cusp types:
	  \subitem $[4k+4, 4k, 4_k]$ and $[(4k+2)_2, 4_k, 2_2]$; alternatively,
	  \subitem $(k+1,2k+1)(4,1)$ and $(2k+1, 4k+4)(2,1)$.

    \item[(e)] A curve of degree $d = 3k+4$, $k \geq 1$ with the following cusp types:
      \subitem $[3k, 3_k]$ and $[4_k, 2_3]$; alternatively,
      \subitem $(k,k+1)(3,1)$ and $(2,2k+1)(2,3)$.
	
	\item[(f)] A curve of degree $d = 14$ with cusp types $[8,4,4,2,2]$ and $[6,6,3,3]$; alternatively, $(2,3)(2,1)(2,1)$ and $(2,5)(3,1)$. 
\end{enumerate}
\end{thm}

\begin{proof}

Recall that we are going to use the same name for a curve, resp. divisor and its strict transform after a blow-up or blow-down.

{\bf For curves under (a1),} consider the following construction. 
Set $C_{l,m} := \{ (y^{l-1}z - x^l)^m - x^{lm-1}y = 0 \} \subset \mathbb{CP}^2$ for any two integers $l \geq 2, m \geq 2$. 
This is a bicupsidal rational projective plane curve. 
Its degree is $d = lm$ and the self-intersection of the minimal good resolution is $\overline{C}^2 = 0$. 
The two singularities are at points $p = [0:1:0]$ and $q = [0:0:1]$.

At $p$ we have a singularity with one Newton pair $\left( m, lm - 1\right)$, or, with multiplicity sequence $[m_{l-1}, m-1]$.
We will need the local intersection multiplicity with its tangent as well, it is $(T_pC_{l,m} \cdot C_{l,m})_p = lm - 1$.

At $q$ we have a singularity with two Newton pairs $(l-1, l)(m, 1)$, or, with multiplicity sequence $[(l-1)m, m_{l-1}]$. 

In what follows, for the sake of simplicity, when considering blow-ups and blow-downs, we will use the same notation for curves and divisors and their strict transforms.
We perform a quadratic Cremona transformation with two proper basepoints (see~\cite[\S 5.3.4]{MoeRat}): one is the transverse intersection point of $C_{l,m}$ and $T_pC_{l,m}$, the other point is $p$. 
The third (non-proper) basepoint is the one infinitely near to $p$, lying at the intersection of the exceptional divisor of the blow-up at $p$ and the strict transform under this blow-up of the line $pq$.
\begin{center}
\begin{figure}[t]
\includegraphics[width=10cm]{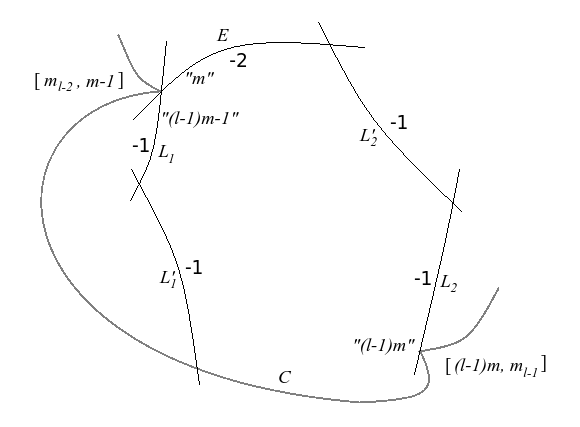}
\caption{The configuration of divisors before blowing down anything at the first step of the construction of curves under (a1).}\label{fig:cremona0}
\end{figure}
\end{center}
\vspace{-3.0em}

More concretely, the following is happening.
Denote by $L_1$ the tangent $T_pC_{l,m}$ and by $L_2$ the line $pq$.
Blow up at $p$ obtaining an exceptional divisor $E$.
Then blow up at the intersection point of (the strict transforms of) the line $L_1$ and the curve $C_{l,m}$ different from the singular point of the curve, obtaining exceptional divisor $L_1'$ and blow up at $E \cap L_2$, obtaining exceptional divisor $L_2'$ (see Figure~\ref{fig:cremona0}).
Notice that this configuration can be blown down in a different way: first blow down $L_1$ and $L_2$, then $E$.
The result is an other bicuspidal curve with longer multiplicity sequences at the cusps.

We repeat similar quadratic Cremona transformation with two proper basepoints. 
We get a three-parameter series ($u \geq 2, l \geq 2, m \geq 2$) of bicuspidal rational curves, with two cusps of topological type as described in the proposition (one quadratic Cremona transformation described above increases $u$ by $1$).

{\bf For curves under (a2),} consider the following construction. 
Set $C_{l,m} := \{ x(yx^l + z^{l+1})^m - z^{(l+1)m + 1} = 0 \} \subset \mathbb{CP}^2$ for any two integers $l \geq 1, m \geq 2$. 
This is a bicuspidal rational projective plane curve. 
Its degree is $d = lm + m + 1$ and the self-intersection of the minimal good resolution is $\overline{C}^2 = 0$. The two singularities are at points $p = [0:1:0]$ and $q = [1:0:0]$.

At $p$ we have a singularity with one Newton pair $\left( lm + 1, (l+1)m + 1\right)$, or, with multiplicity sequence $[lm + 1, m_l]$.

At $q$ we have a singularity with one Newton pair $\left( m, (l+1)m + 1\right)$, or, with multiplicity sequence $[m_{l+1}]$. 
We will need the local intersection multiplicity with its tangent as well, it is $(T_qC_{l,m} \cdot C_{l,m})_q = (l+1)m$.

We perform a quadratic Cremona transformation with two proper basepoints: one is the transverse intersection point of $C_{l,m}$ and $T_qC_{l,m}$, the other point is $q$. 
The third (non-proper) basepoint is the one infinitely near to $q$, lying at the intersection of the exceptional divisor of the blow-up at $q$ and the strict transform under this blow-up of the line $pq$.

We repeat similar quadratic Cremona transformation with two proper basepoints. 
We get a three-parameter series ($u \geq 2, l \geq 1, m \geq 2$) of bicuspidal rational curves, with two cusps of topological type as described in the proposition (one quadratic Cremona transformation described above increases $u$ by $1$).

{\bf For curves under (a3),} consider the following construction. 
Set $C_{l,m} := \{ x(yx^l + z^{l+1})^m - z^{(l+1)m + 1} = 0 \} \subset \mathbb{CP}^2$ for any two integers $l \geq 1, m \geq 2$. 
This is a bicuspidal rational projective plane curve. 
Its degree is $d = lm + m + 1$ and the self-intersection of the minimal good resolution is $\overline{C}^2 = 0$. 
The two singularities are at points $p = [0:1:0]$ and $q = [1:0:0]$.

At $p$ we have a singularity with one Newton pair $\left( lm + 1, (l+1)m + 1\right)$, or, with multiplicity sequence $[lm + 1, m_l]$. 
We will need the local intersection multiplicity with its tangent as well, it is $(T_pC_{l,m} \cdot C_{l,m})_p = (l+1)m+1$.

At $q$ we have a singularity with one Newton pair $\left( m, (l+1)m + 1\right)$, or, with multiplicity sequence $[m_{l+1}]$. 
We will need the local intersection multiplicity with its tangent as well, it is $(T_qC_{l,m} \cdot C_{l,m})_q = (l+1)m$.

We perform a quadratic Cremona transformation with two proper basepoints: one is the smooth transverse intersection point of $C_{l,m}$ and the tangent line $L_q = T_qC_{l,m}$, the other point is the intersection point $L_q \cap L_p$ of the tangents at singularities ($L_p = T_pC_{l,m}$).
The third (non-proper) basepoint is the one infinitely near to $L_p \cap L_q$, lying at the intersection of the exceptional divisor of the blow-up at $L_p \cap L_q$ and the strict transform under this blow-up of the line $L_p$.

We repeat similar quadratic Cremona transformation with two proper basepoints.
We get a three-parameter series ($u \geq 2, l \geq 1, m \geq 2$) of bicuspidal rational curves, with two cusps of topological type as described in the proposition (one quadratic Cremona transformation described above increases $u$ by $1$).

{\bf For curves under (a4),} consider the following construction. 
Set $C_{l,m} := \{ (y^{l-1}z - x^l)^m - x^{lm-1}y = 0\} \subset \mathbb{CP}^2$ for any two integers $l \geq 2, m \geq 2$. 
This is a bicuspidal rational projective plane curve. 
Its degree is $d = lm$ and the self-intersection of the minimal good resolution is $\overline{C}^2 = 0$. 
The two singularities are at points $p = [0:1:0]$ and $q = [0:0:1]$.

At $p$ we have a singularity with one Newton pair $(m, lm - 1)$, or, with multiplicity sequence $[m_{l-1}, m-1]$. 
We will need the local intersection multiplicity with its tangent as well, it is $(T_pC_{l,m} \cdot C_{l,m})_p = lm - 1$.

At $q$ we have a singularity with two Newton pairs $(l-1, l)(m, 1)$, or, with multiplicity sequence $[(l-1)m, m_{l-1}]$.
We will need the local intersection multiplicity with its tangent as well, it is $(T_qC_{l,m} \cdot C_{l,m})_q = lm$.

We perform a quadratic Cremona transformation with two proper basepoints: one is the transverse intersection point of $C_{l,m}$ and the tangent line $L_p = T_pC_{l,m}$, the other point is the intersection point $L_p \cap L_q$ of the tangents at singularities ($L_q = T_qC_{l,m}$). 
The third (non-proper) basepoint is the one infinitely near to $L_p \cap L_q$, lying at the intersection of the exceptional divisor of the blow-up at $L_p \cap L_q$ and the strict transform under this blow-up of the line $L_q$.

We repeat similar quadratic Cremona transformation with two proper basepoints. 
We get a three-parameter series ($u \geq 2, l \geq 2, m \geq 2$) of bicuspidal rational curves, with two cusps of topological type as described in the proposition (one quadratic Cremona transformation described above increases $u$ by $1$).

{\bf For curves under (b),} take a configuration of two conics $C_1$ and $C_2$ and a line $L$ as follows: $C_1$ and $C_2$ has only one intersection point $p = C_1 \cap C_2$ of local intersection multiplicity $4$; $C_1$ and $L$ has also one intersection point only, namely $q = C_1  \cap L$ which is a touching point with local intersection multiplicity $2$; further, $\{ r, s \} = C_2 \cap L$ are two distinct transverse intersection points.

Such a configuration exists, consider for example the set of curves given by equations $\{ \{x^2 + y^2 - yz = 0\}, \{x^2 - y^2 - yz = 0\}, \{y - z = 0\} \}$ in $\CP^2$ equipped with homogeneous coordinates $[x:y:z]$.

Again, for simplicity, we will use the same notation for curves, respectively divisors and their strict transforms. 
Blow up $p$ obtaining an exceptional divisor $E_1$, then blow up the intersection point of (the strict transforms of) $C_1$ and $C_2$ three more times, obtaining exceptional divisors $E_2, E_3, E_4$ in this order.
Then blow up the point $C_1 \cap E_4$ resulting in exceptional divisor $C_1'$ and blow up $s$, one of the intersection points of $L$ and $C_2$, resulting in exceptional divisor $C_2'$.

Notice that one can now blow down $C_1$, $C_2$, $E_4, E_3, E_2$ and $E_1$ in this order, then the strict transform of $L$ will be a unicuspidal rational curve with cusp type $[2]$ at a point to be called $q$ from now on, $C_1'$ a conic going through the cusp at $q$ and touching $L$ at a point to be called $r$ from now on with local intersection multiplicity $4$, $C_2'$ a conic going through $r$ having local intersection multiplicity $5$ with $L$ at $r$ and a transverse intersection point  with $L$ at some other, smooth point which we will call $s$.

Now one can repeat a similar process: rename $C_1'$ to $C_1$, $C_2'$ to $C_2$, blow up four times the intersection point of (the strict transforms of) $C_1$ and $C_2$ (keeping calling it $r$ at each step) resulting in exceptional divisors $E_1$, $E_2$, $E_3$, $E_4$ in this order, then blow up at $C_1 \cap E_4$ obtaining $C_1'$. 
Now further blow up at the intersection point of $L$ and $C_2$ not lying on $E_4$.
\begin{center}
\begin{figure}[b]
\includegraphics[width=12cm]{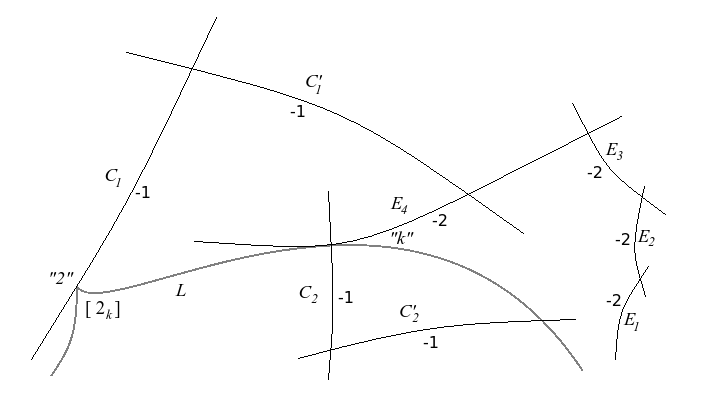}
\caption{The configuration of divisors after blowing up and before blowing down at the intermediate step of constructing curves of type as listed in (b).}\label{fig:cremona1}
\end{figure}
\end{center}
\vspace{-3.0em}

After this one has the following configuration (see Figure~\ref{fig:cremona1}) for $k = 1$ (in a complex surface obtained from $\mathbb{CP}^2$ by the blow-ups):
A curve $L$ with one cusp of type $[2_k]$ at a point to be called $q$, a divisor $C_1$ going through $q$ (having intersection multiplicity $2$ with $L$) and intersecting $C_1'$, $C_1'$ intersecting $E_4$, $E_4$ touching $L$ at a point to be called $r$ with local intersection multiplicity $k$, $C_2$ going through $r$ and intersecting $C_2'$, $C_2'$ intersecting $L$ in a point to be called $s$, $E_i$ intersecting $E_{i+1}$ for $i = 1,2,3$.
Divisors $C_1, C_1', C_2, C_2'$ have self-intersection $-1$, $E_i$, $i = 1,2,3,4$ have self-intersection $-2$.

Notice that this configuration can be blown down to get $\mathbb{CP}^2$ in two different ways: blowing down $C_1'$, $C_2'$, $E_4$, $E_3$, $E_2$, $E_1$ in this order or blowing down $C_1$, $C_2$, $E_4$, $E_3$, $E_2$, $E_1$ in this order.
In the first case, the strict transform of $L$ is a bicuspidal curve with cusp types $[k_4]$ (smooth for the degenerating starting case $k=1$) and $[2_k]$, in the second case it is a bicuspidal curve with cusp types $[(k+1)_4]$ and $[2_{k+1}]$.

This, together with the construction of the staring case $k=1$ above, gives an inductive construction of curves claimed in (b).

{\bf For curves under (c),} consider the following configuration of two conics $C_1$ and $C_2$: $\{p, q\} = C_1 \cap C_2$ such that $p$ is a point with local intersection multiplicity $3$ and $q$ is a transverse intersection point, $L$ is a line tangent to $C_1$ at point $r$ and tangent to $C_2$ at point $s$.

Such a configuration exists, consider for example the set of curves given by equations $\{\{x^2 + y^2 + xy - yz = 0\}, \{x^2 + y^2 - xy - yz = 0\}, \{3y - 4z = 0\}\}$.

Now blow up three times the intersection point of (the strict transforms of) $C_1$ and $C_2$ distinct from $q$, resulting in exceptional divisors $E_1, E_2, E_3$, in this order.
Blow up at point $q$ as well, obtaining the exceptional divisor $E_0$.

Blow up $E_3 \cap C_1$, obtaining $C_1'$ and $E_0 \cap C_2$, obtaining $C_2'$.
Notice that one can blow down $C_1, C_2$, $E_0$, $E_3$, $E_2$, $E_1$ in this order to obtain a curve with $2$ cusps of type $[2_4]$ at point called $p$ and $[2_2]$ at point called $q$, respectively, such that $C_1'$ and $C_2'$ are both going through the cusps at $p$ and $q$ and have no further intersection points with the bicuspidal curve $L$ (the strict transform of the line originally called $L$).
$C_1'$  with $L$ has local intersection multiplicity $6$ at $p$ and $4$ at $q$; $C_2'$ with $L$ has local intersection multiplicity $8$ at $p$ and $2$ at $q$.

Now rename $C_1'$ to $C_1$ and $C_2'$ to $C_2$ and repeat a similar process: blow up at the intersection point of (strict transforms of) $C_1$ and $C_2$ distinct from $q$ three times, obtaining $E_1, E_2, E_3$, respectively, then blow up at $q$ obtaining $E_0$.
Blow up $C_1 \cap E_3$ obtaining $C_1'$ and $C_2 \cap E_0$ obtaining $C_2'$.
Notice that this leads to a configuration (see Figure~\ref{fig:cremona2}) with $k = 1$ as follows:
$E_i$ is intersecting $E_{i+1}$ for $i=1,2$, $E_3$ is intersecting $E_2$, $C_1'$ and $C_2$, $E_0$ is intersecting $C_1$ and $C_2'$, $C_i$ is intersecting $C_i'$ for $i=1,2$.
$L$ has a cusp of type $[2_k]$ at $C_2 \cap E_3$, $L$ having local intersection multiplicity $2k$ with $E_3$ and $2$ with $C_2$ at that point; $L$ has a cusp type $[2_k]$ at $C_1 \cap E_0$ as well, $L$ having local intersection multiplicity $2k$ with $E_0$ and $2$ with $C_1$ at that point.
\begin{center}
\begin{figure}[t]
\includegraphics[width=12cm]{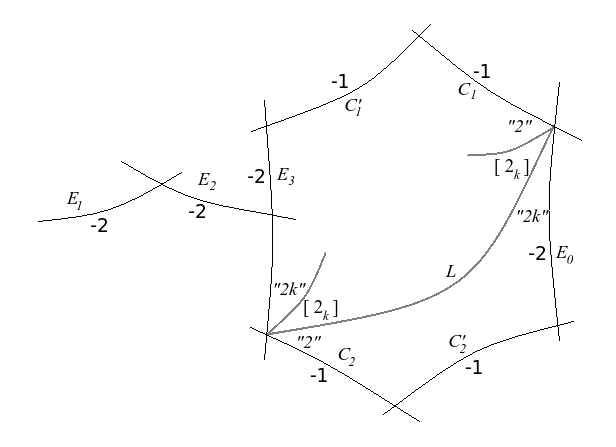}
\caption{The configuration of divisors after blowing up and before blowing down at the intermediate step of constructing curves of type as listed in (c).}\label{fig:cremona2}
\end{figure}
\end{center}
\vspace{-3.0em}

Notice that one can blow down this configuration in two different ways to obtain $\mathbb{CP}^2$ as ambient space: $C_1'$, $C_2'$, $E_0$, $E_3, E_2, E_1$ in this order, or $C_1$, $C_2$, $E_0$, $E_3, E_2, E_1$ in this order.
The first option takes $L$ to a bicuspidal curve with cusp types $[(2k)_3, 2_k]$ and $[(2k), 2_k]$, the second takes $L$ to a bicuspidal curve with cusp types $[(2k+2)_3, 2_{k+1}]$ and $[(2k+2), 2_{k+1}]$.
In this way, this construction together with obtaining the starting case $k=1$ as above, leads to an inductive construction of curves as claimed in (c).

{\bf For curves under (d1) and (d2),} one can check that these are described in~\cite[Main Theorem, (s)]{BorZolCom}, where a parametrization is given for them.

As mentioned in~\cite{BorZolCom}, according to M. Koras, this series was first constructed by Pierrette Cassou-Nogu\`{e}s.
 One can obtain these curves recursively by birational transformation as follows.
 We just sketch the construction.
 Start with a conic $C$ and a line $L$ in $\mathbb{CP}^2$ intersecting in two distinct points $p$ and $q$.
 Blow up at point $p$ to obtain $E_1$, then blow up at point $E_i \cap C$ to obtain $E_{i+1}$ for $i=1,2,3$, then blow up at $L \cap C$ to obtain $F$.
 This configuration can be blown down in an other way to get $\mathbb{CP}^2$ as ambient space again: first blow down $C$, then $L$, then $E_1, E_2, E_3$ in this order.
 The strict transform of $E_4$ will be a line to be renamed to $L$ and the strict transform of $F$ will be a conic to be renamed to $C$.
 This is the birational transformation to obtain $D_k$ from $C_k$ and $C_{k+1}$ from $D_k$.
 The starting configuration is a conic $C'$ touching $C$ in a point $r \neq p,q$ with local intersection multiplicity $4$  and $L$ in a point $s \neq p,q$ with local intersection multiplicity $2$.
 
{\bf Curves under (e)} correspond to those given in~\cite[Main Theorem, (i)]{BorZolCom}.
To give a series of birational transformations producing them, consider the following.
Equation $y^4-2xy^2z+x^2z^2-yz^3=0$ determines a rational unicuspidal curve of degree $4$ and cusp type $[2_3]$ (cf.~\cite[\S 3.2, curve $C_3$]{MoeCus}).
Let $p$ be a smooth inflection point on the curve and $L_1$ be the tangent to the curve at $p$, with local intersection multiplicity $3$.
Let $r$ be the other intersection point of $L_1$ with the curve.
Let $L_2$ be the line having local intersection multiplicity $4$ with the curve at its cusp at point called $q$.
$L_2$ has no other intersection point with the curve.
Let $s$ be the intersection point of $L_1$ and $L_2$.
Blow up $s$ to produce exceptional divisor $E$, then blow up $r$ to produce $L_1'$ and $E \cap L_2$ to produce $L_2'$.
Now blow down $L_1$ and $L_2$, then $E$.
The strict transform of the curve will be a bicuspidal curve as described in (e) for $k=1$.
Rename $L_1'$ to $L_1$ and $L_2'$ to $L_2$ and repeat a similar process, that is, perform a quadratic Cremona transformation with two proper basepoints: one proper basepoint being the point at the cusp of type $[3k, 3_k]$, the other proper basepoint being the transverse intersection point of $L_1$ and the curve and the third, non-proper basepoint being the one infinitely near to the first basepoint, lying at the intersection of (the strict transform of) $L_2$ and the exceptional divisor produced by the blow-up of the first cusp (the one of type $[3k, 3_k]$).
After this process, we get a curve with cusp types $[3(k+1), 3_{k+1}]$ and $[4_{k+1}, 2_3]$ and with local intersection multiplicities with two lines as needed to continue the induction.
 
{\bf For the curve in (f)}, a parametrization (by $[t:s] \in \CP^1$) is given in~\cite[Main Theorem, (t)]{BorZolCom}:
\[ x = 3t^{14} + 3st^{13} + 2s^2t^{12}, \quad y = 3s^{12}t^2 - 3s^{13}t + s^{14}, \quad z = 3s^6t^8. \]
 
\end{proof}

\begin{rem}\label{rem:notice}
Tono in \cite{TonOna} and Fenske in \cite{FenRat1} constructed some rational bicuspidal curves.

From curves under (a1), taking $u = 2$, we get Tono's first bicuspidal 2-parameter series with $\overline{C}^2 = -1$ ($a_{\textrm{Tono}} = l-1$,~\cite[Theorem 2, No. 1]{TonOna}).
Taking $l=2$, we get Fenske's 6th series ($a_{\textrm{Fenske}} = u, d_{\textrm{Fenske}} = m - 1$,~\cite[Theorem 1.1, 6]{FenRat1}).

From curves under (a2), taking $u = 2$ we get Tono's second bicuspidal 2-parameter series with $\overline{C}^2 = -1$ ($a_{\textrm{Tono}} = l$,~\cite[Theorem 2, No. 2]{TonOna}).
For $l=1$, we get Fenske's 5th series ($d_{\textrm{Fenske}} = m, a_{\textrm{Fenske}} = u$,~\cite[Theorem 1.1, 5]{FenRat1}).

From curves under (a3), taking $u = 2$ we get Tono's third bicuspidal 2-parameter series with $\overline{C}^2 = -1$ ($a_{\textrm{Tono}} = l+1$,~\cite[Theorem 2, No. 3]{TonOna}).
For $l = 0$, we get Fenske's 4th series ($d_{\textrm{Fenske}} = m, a_{\textrm{Fenske}} = u - 1$,~\cite[Theorem 1.1, 4]{FenRat1}).

From curves under (a4), taking $u = 2$ we get Tono's 4th bicuspidal 2-parameter series with $\overline{C}^2 = -1$ ($a_{\textrm{Tono}} = l-1$,~\cite[Theorem 2, No. 4]{TonOna}).
In Fenske's terminology~\cite{FenRat}, a cuspidal curve $C$ is of type $(d, d-k)$ if $\textrm{deg}(C) = d$ and the maximal multiplicity of the cusps is $d-k$.
Flenner, Zaidenberg and Fenske classified all rational cuspidal curves with $k = 2, 3$ and, assuming the rigidity conjecture, also for $k = 4$, see~\cite{FenRat1},~\cite{FenRat},~\cite{FleZaiOna},~\cite{FleZaiRat}.  For this series (a4), $k = \textrm{min} \{lm, (u-1)(lm-1)\}$.

Curves under (a1)-(a4) should also be compared with~\cite[Main Theorem, (c)-(f)]{BorZolCom}.

We did not mention in the above list curves in~\cite[Main Theorem, (a)]{BorZolCom}.
Notice that those are of Lin--Zaidenberg type (that is, there exists a line having exactly one intersection point with the curve).
Curves of Lin--Zaidenberg type are completely described in~\cite{TonDef, ZaiLinAn}.
The complement $\mathbb{CP}^2 \setminus C$ of these curves is of logarithmic Kodaira dimension $1$.
For explicit description of local topological types of their singularities see \cite[\S 8]{FLMNOnrves}.

\end{rem}

\end{document}